\documentclass[a4paper,11pt]{article}

\usepackage[scale=0.7,vmarginratio={1:2},heightrounded]{geometry}

\usepackage[utf8]{inputenc}

\usepackage{amsmath}
\usepackage{amssymb}
\usepackage{amsfonts}
\usepackage{amsthm}
\usepackage{enumitem}
\usepackage{xcolor}
\usepackage{authblk}

\usepackage{parskip}

\begingroup
    \makeatletter
    \@for\theoremstyle:=definition,remark,plain\do{%
        \expandafter\g@addto@macro\csname th@\theoremstyle\endcsname{%
            \addtolength\thm@preskip\parskip
            }%
        }
\endgroup

\newcommand{\tr}{\operatorname{tr}}
\newcommand{\adj}{\operatorname{adj}}

\newcommand{\R}{{\mathbb R}}
\newcommand{\C}{{\mathbb C}}
\newcommand{\Mcal}{{\mathcal M}}

\newcommand{\PD}{\text{PD}}
\newcommand{\PP}{\mathbb{P}}

\newtheorem{theorem}{Theorem}
\newtheorem{lemma}[theorem]{Lemma}
\newtheorem{prop}[theorem]{Proposition}
\newtheorem{cor}[theorem]{Corollary}

\theoremstyle{definition}
\newtheorem{definition}[theorem]{Definition}
\newtheorem{example}[theorem]{Example}
\newtheorem{remark}[theorem]{Remark}
\newtheorem*{conj}{Conjecture}
\newtheorem*{prob}{Problem}

\numberwithin{theorem}{section}

\newcommand{\X}{\mathcal{X}}
\newcommand{\Y}{\mathcal{Y}}

\newcommand{\fc}{h}
\newcommand{\gc}{k}
\newcommand{\ind}{m}
\newcommand{\kval}{M}

\title{Maximum likelihood degree of the two-dimensional linear Gaussian covariance model}

\author{Jane Ivy Coons, Orlando Marigliano, Michael Ruddy}

\date\today

\begin{document}
\maketitle

\begin{abstract}
In algebraic statistics, the maximum likelihood degree of a statistical model is
the number of complex critical points of its log-likelihood function.
A priori knowledge of this number is useful
for applying techniques of numerical algebraic geometry
to the maximum likelihood estimation problem.
We compute the maximum likelihood degree of a generic two-dimensional subspace of the space of $n\times n$ Gaussian covariance matrices. We use the intersection theory of plane curves to show that this number is $2n-3$.
\end{abstract}

\section{Introduction}
A linear Gaussian covariance model is a collection of multivariate Gaussian probability distributions whose covariance matrices are linear combinations of some fixed symmetric matrices. In this paper, we will focus on the \emph{two-dimensional linear Gaussian covariance model}, in which all of the covariance matrices in the model lie in a two-dimensional linear space. Linear Gaussian covariance models were first studied by Anderson in \cite{anderson1970} in the context of the analysis of time series models. They continue to be studied towards this end, for example, in \cite{wu2012}. These models also have applications in a variety of other contexts.

One of the most common types of linear Gaussian covariance models consist of covariance matrices with some prescribed zeros. Given a Gaussian random vector $(X_1, \dots, X_n)$ with mean $\mu$ and positive definite covariance matrix $\Sigma \in \R^{n \times n}$, we can discern independence statements from the zeros in $\Sigma$. In particular, the disjoint subvectors $(X_{i_1}, \dots, X_{i_k})$ and $(X_{j_1}, \dots, X_{j_l})$ are independent if and only if the submatrix of $\Sigma$ that consists of rows $i_1, \dots, i_k$ and columns $j_1, \dots, j_l$ is the zero matrix \cite[Proposition~2.4.4]{sullivant2018}. 

Maximum likelihood estimation for covariance matrices with a fixed independence structure was studied in \cite{chaudhuri2007}. These types of models find applications in the study of gene expression using relevance networks \cite{butte2000}. In these networks, genes are connected with an edge if their expressions are sufficiently correlated. The edges and non-edges in the resulting graph dictate the sparsity structure of the covariance matrix. Problems related to estimation of sparse covariance matrices have been studied in \cite{bien2011} and \cite{rothman2010}.

Linear Gaussian covariance models are also applicable to the field of phylogenetics. In particular, Brownian motion tree models, which model evolution of normally distributed traits along an evolutionary tree, are linear Gaussian covariance models \cite{felsenstein1973}. The covariance matrices of Brownian motion tree models require linear combinations of more than two matrices. However, the authors believe that the results in this paper will find applications to mixtures of Brownian motion tree models. These apply, for example, to models of trait evolution that consider two genes instead of just one \cite{jiang2017}.

Algorithms for computing the maximum likelihood estimate for generic linear Gaussian covariance models have been the subject of much study \cite{anderson1970, anderson1973, bien2011, chaudhuri2007}. Zwiernik, Uhler and Richards have shown that when the number of data points is sufficiently large, maximum likelihood estimation for such models behaves like a convex optimization problem in a large convex region containing the maximum likelihood estimate \cite{zwiernik2017}.

In this paper, we are concerned with computing the maximum likelihood degree of the two-dimensional linear Gaussian covariance model for generic parameters and data. This is the number of complex critical points of the log-likelihood function, and it is considered to be a measurement of the difficulty of computing the maximum likelihood estimate \cite[Table~3]{sturmfels2019}. Knowledge of the ML-degree of a model is important when applying numerical algebraic geometry methods to solve the MLE problem; in particular, it gives a stopping criterion for monodromy methods \cite[Section 5]{sturmfels2019}. For more background on ML-degrees, we refer the reader to \cite{bernd-maximum-likelihood-degree} and \cite[Chapter 2]{drton2008}.

\section{Preliminaries}

Let $n$ be a natural number, and let $\PD_n\subset \mathbb R^{\binom{n+1}{2}}$ be the cone of all $n\times n$ symmetric positive definite matrices. We view $\PD_n$ as the space of covariance matrices of all normal distributions $\mathcal N(0,\Sigma)$ with mean zero.

In algebraic statistics, a Gaussian statistical model is an algebraic subset of $\PD_n$. In this paper, we consider models of the form
\[
\mathcal M_{A,B} = \{xA + yB \mid x,y\in \mathbb R\} \cap \PD_n
\]
for symmetric matrices $A$ and $B$, whenever the intersection is not empty. That is, $\mathcal M_{A,B}$ is the intersection of the positive definite cone with the linear span of $A$ and $B$. We call $\mathcal M_{A,B}$ the \emph{two-dimensional linear Gaussian covariance model} with respect to $A$ and $B$.

Given independent, identically distributed (i.i.d.)\ samples $u_1, \dots, u_r \in \R^n$ from some normal distribution, the maximum likelihood estimation problem for $\mathcal M_{A,B}$ is to find a covariance matrix $\hat\Sigma\in \Mcal_{A,B}$, if one exists, that maximizes the value of the likelihood function
\[
	L(\Sigma\mid u_1,\dotsc, u_r) = \prod_{i=1}^r f_{\Sigma}(u_i),
\]
where $f_\Sigma$ is the density of $\mathcal N(0,\Sigma)$. Let $S$ denote the \emph{sample covariance matrix}
\[
S = \frac{1}{r} \sum_{i=1}^r u_i u_i^T.
\]
Since for all $\Sigma$ the value $L(\Sigma\mid u_1,\dotsc, u_r)$ only depends on $S$, we identify the data given by $r$ i.i.d.\ samples from a normal distribution with their sample covariance matrix $S$. The logarithm is a concave function, so the maximizer of the likelihood function is also the maximizer of its natural log, the \emph{log-likelihood function}. This function can be written in terms of $S$:
\begin{align*}
\ell(\Sigma \mid S) & := \log L(\Sigma \mid S)\\
&= -\frac{rn}{2} \log(2 \pi) - \frac{r}{2} \log \det (\Sigma) -\frac{r}{2} \text{tr}(S \Sigma^{-1}).
\end{align*}
Note that the maximizer of this function is equal to the minimizer of
\[
\tilde\ell(\Sigma \mid S) := \log \det (\Sigma) + \text{tr}(S \Sigma^{-1}).
\]
When we restrict to the model $\Mcal_{A,B}$, we require that $\Sigma = xA + yB$ for some $x,y \in \mathbb{R}$ such that $xA + yB$ is positive definite. So the maximum likelihood estimation problem in this case is equivalent to
\[
\begin{aligned}
& \underset{x,y}{\text{argmin}}
& & \tilde\ell(xA + yB \mid S) \\
& \text{subject to}
& & xA + yB \in \PD_n.
\end{aligned}
\]
To find local extrema of the log-likelihood function, we set its gradient equal to $0$ and solve for $x$ and $y$. The two resulting equations are called the score equations. 

\begin{definition}\label{def:mle}
The \emph{score equations} for $\mathcal M_{A,B}$ are the partial derivatives of the function $\tilde\ell(xA+yB \mid S)$ with respect to $x$ and $y$. The \emph{maximum likelihood degree} or \emph{ML-degree} of $\mathcal M_{A,B}$ is the number of complex solutions to the score equations, counted with multiplicity, for a generic sample covariance matrix $S$.
\end{definition}

Definition~\ref{def:mle} makes reference to a \emph{generic} sample covariance matrix. We give a detailed explanation of this term from algebraic geometry at the end of this section.

One benefit of working with $\tilde\ell$ is that the score equations are rational functions of the data. This allows us to use tools from algebraic geometry to analyze their solutions. Let $\Sigma = xA + yB$. For the sake of brevity, we will denote $P(x,y) = \det \Sigma$ and $T(x,y) = \tr (S \adj \Sigma)$, where $\adj \Sigma$ is the classical adjoint.
With this notation, the function $\tilde\ell$ takes the form
\[
	\tilde\ell(\Sigma\mid S) = \log P + \frac{T}{P}.
\]
Accordingly, the score equations are
\begin{align*}
\tilde\ell_x(x,y) &= \frac{P_x}{P} + \frac{P T_x - T P_x}{P^2} \\
\tilde\ell_y(x,y) &= \frac{P_y}{P} + \frac{P T_y - T P_y}{P^2}.
\end{align*}
Here and throughout, the notation $h_x$ is used for the derivative of a function $h$ with respect to the variable $x$.
We are concerned with values of $(x,y) \in \C^2$ where both of the score equations are zero. We clear denominators by multiplying $\tilde\ell_x$ and $\tilde\ell_y$ by $P^2$ to obtain two polynomials,
\begin{align}\label{Eqn:fg}
f(x,y) &:= P P_x + P T_x - T P_x \nonumber \\
g(x,y) &:= P P_y + P T_y - T P_y.
\end{align}
We note the \emph{degrees} of each relevant term for generic $A$, $B$ and $S$. Specifically, their total degree with respect to their variables $x$ and $y$ are:
\begin{align*}
\deg P &= n \\
\deg P_x = \deg P_y = \deg T &= n-1 \\
\deg T_x = \deg T_y & = n-2.
\end{align*}
A polynomial $h$ is called a \emph{homogeneous form} if each of its terms has the same degree. The polynomials $f$ and $g$ can be written as a sum of a homogeneous degree $2n-1$ form with a homogeneous degree $2n-2$ form. 

The critical points of $\tilde\ell$ are in the variety $V(f,g)$. However, this variety also contains points at which $\tilde\ell$ and the score equations are not defined since we cleared denominators. The ideal whose variety is exactly the critical points of $\tilde\ell$ is the saturation
\begin{align*}
J & = \mathcal{I}(f,g) : \langle P \rangle^{\infty} \\
&:= \{h \in \C[x,y] \mid h P^N \in \mathcal{I}(f,g) \text{ for some } N\}
\end{align*}
Saturating with $P = \det \Sigma$ removes all points in $V(f,g)$ where the determinant is zero and $\tilde\ell$ is undefined. For more details on the geometric content of saturation, we refer the reader to Chapter 7 of \cite{sullivant2018}.  We will show that $\mathcal{I}(f,g)$ and hence $J$ are zero-dimensional in Lemmas \ref{Lem:Irreducible} and \ref{Lem:NotConstantMults}. The ML-degree of the model is hence the degree of $J$. This is the number of isolated points in the variety of $J$ counted with multiplicity. For more background on degrees of general varieties,  see \cite[Lec.\ 13]{harris} and \cite[Ch.\ 4, Sec.\ 1.4]{shafarevich2013}.

We now state the main result and offer an outline for its proof, which we follow in the remaining sections.

\begin{theorem}\label{Thm:Main}
For generic $n \times n$ symmetric matrices $A$ and $B$, the maximum likelihood degree of the two-dimensional linear Gaussian covariance model $\Mcal_{A,B}$ is $2n-3$.
\end{theorem}


A key tool used in the proof of Theorem \ref{Thm:Main} is B\'ezout's theorem, a proof of which can be found in Chapter 5.3 of \cite{fulton1989}.

\begin{theorem}[B\'ezout's Theorem] \label{Thm:Bezout}
Let $H$ and $K$ be projective plane curves of degrees $d_1$ and $d_2$ respectively. 
Suppose further that $H$ and $K$ share no common component. 
Then the intersection of $H$ and $K$ is zero-dimensional and
the number of intersection points of $H$ and $K$, counted with multiplicity, is $d_1 d_2$.
\end{theorem}

Let $F(x,y,z)$ and $G(x,y,z)$ denote the homogenizations of $f$ and $g$ with respect to $z$. 
Then $F$ and $G$ both define projective plane curves of degree $2n-1$. 
Lemmas~\ref{Lem:Irreducible} and~\ref{Lem:NotConstantMults}
will show that $F$ and $G$ do not share a common component.
So we can apply B\'ezout's Theorem to count their intersection points.

Let $q=[x:y:z]$ be a point in $\C\PP^2$. Then by B\'ezout's theorem,
$$
(2n-1)^2=\sum_{q\in V(F,G)} I_q(F,G),
$$
where $I_q(F,G)$ denotes the \emph{intersection multiplicity} of $F$ and $G$ at $q$.
The definition and properties of the intersection multiplicity of a pair of algebraic curves at a point can be found in~\cite[Sec.\ 3, Thm.\ 3]{fulton1989}. For affine points $(x,y)\in V(f,g)$ we sometimes denote the intersection multiplicity as $I_{(x,y)}(f,g):=I_{[x:y:1]}(F,G)$.

We show in Proposition \ref{Prop:SaturateOrigin} that saturating the ideal $\mathcal{I}(f,g)$ with $\det \Sigma$ corresponds to removing only the origin from the affine variety of $f$ and $g$.
This in turn corresponds to removing the point
$[0{\,:\,}0{\,:\,}1]$
from the projective variety $V(F,G)$.
Since we are only interested in affine intersection points of $F$ and $G$ outside of the origin, we split the sum on the right-hand side of the above equation as follows:
\begin{equation}\label{Eq:IntMultSum}
(2n-1)^2=I_{[0:0:1]}(F,G)+
\sum_{\substack{q\in V(F,G)\\
q\notin \{[0:0:1]\} \cup V(F,G,z)}} I_q(F,G)
+\sum_{q\in V(F,G,z)} I_q(F,G).
\end{equation}
The middle term of the right-hand side of \eqref{Eq:IntMultSum} is exactly the degree of the saturated ideal $J = \mathcal{I}(f,g) : \langle \det \Sigma \rangle^{\infty}$.
Thus one can find the degree of $J$ by computing the intersection multiplicities of $F$ and $G$ at the origin and at their intersection points at infinity. We compute the former in Section~\ref{Sec:MultiplicityOfTheOrigin} and the latter in Section~\ref{Sec:MultiplicityAtInfinity} to obtain
\[
I_{[0:0:1]}(F,G)=(2n-2)^2 \quad \text{and}\quad 
\sum_{q\in V(F,G,z)} I_q(F,G)=2n
\]
for generic $A, B$ and $S$.
Thus, by rearranging \eqref{Eq:IntMultSum},
\begin{align*}
\sum_{\substack{q\in V(F,G)\\q\notin \{[0:0:1]\}\cup V(F,G,z)}}
I_{q}(F,G)
&=(2n-1)^2-(2n-2)^2-2n = 2n-3,
\end{align*}
which implies $\deg(J)=2n-3$.

\begin{example}\label{expl}
Let $n=3$ and consider the model $\mathcal M_{A,B}$ defined by the positive definite matrices
\begin{align*}
A = \begin{pmatrix}
5 & 1 & 0 \\
 1 & 3 & -2 \\
0 & -2 & 6
\end{pmatrix}, \quad
B = \begin{pmatrix}
1 & -1 & 0 \\
-1 & 6 & -2 \\
0 & -2 & 1
\end{pmatrix}.
\end{align*}
Using the \textit{Julia} software package \texttt{LinearCovarianceModels.jl} \cite{sturmfels2019} we find that the maximum likelihood degree of $\mathcal{M}_{A,B}$ is indeed $2\cdot 3-3 = 3$, meaning that for a generic sample covariance matrix there will be three solutions over the complex number to the score equations. If we take the sample covariance matrix,
\begin{align*}
S &= \begin{pmatrix}
1 & 2 & -2 \\
2 & 6 & -7 \\
-2 & -7 & 9
\end{pmatrix},
\end{align*}
then the equations $f,g$ from \eqref{Eqn:fg} for $\mathcal M_{A,B}$ and $S$ are given explicitly by
\begin{align*}
f = 12288x^5+57600x^4y+74272x^3y^2+20172x^2y^3+1729xy^4+37y^5\\
-10496x^4-33792x^3y-45484x^2y^2-7232xy^3-513y^4
\end{align*}
and
\begin{align*}
g = 11520x^5+37136x^4y+20172x^3y^2+3458x^2y^3+185xy^4+3y^5\\
-12624x^4-9448x^3y-6480x^2y^2-528xy^3-21y^4.
\end{align*}
We used the numerical polynomial solver package \texttt{HomotopyContinuation.jl} \cite{Jul:HC} to find the solutions to the system of equations $f=0,g=0$. The solution set consisted of the origin (with multiplicity 16) and three points corresponding to the critical points of the log-likelihood function,
\begin{align*}
\{ (0.6897, 0.1773), (0.2655 + &0.3071 i, 0.9865 - 2.4601 i),\\
&(0.2655 - 0.3071 i, 0.9865 + 2.4601 i)\}.
\end{align*}
The number of critical points and the multiplicity at the origin are predicted by Theorem~\ref{Thm:Main} and Corollary~\ref{Thm:OriginMult} respectively. This fits into Equation~\eqref{Eq:IntMultSum}, which for $n=3$ becomes $5^2 = 16 + 3 + 6$.
Thus the maximum likelihood estimate for $\mathcal{M}_{A,B}$ and $S$ is the real point in the list above, which corresponds to the positive definite covariance matrix
\[
\Sigma = \begin{pmatrix}
 3.6257 &  0.5124 & 0 \\
0.5124 & 3.1329 & -1.7340 \\
0 &  -1.7340 & 4.3154
\end{pmatrix}
\]
that maximizes the likelihood function $L(\Sigma\mid S)$.
\end{example}

\paragraph{Properties that hold generically.}\label{generic-properties}

In Example~\ref{expl}, it was important to choose the matrices $A,B,S$ to be ``generic enough.'' We explain the precise notion of genericity in classical algebraic geometry below.

Let $X$ be an algebraic variety and $\mathcal P$ a property of the points of $X$. One says that $\mathcal P(x)$ \emph{holds for generic} $x\in X$, or \emph{holds generically} on $X$, if there exists a non-empty Zariski open set $U$ of $X$ such that $\mathcal P(x)$ holds for all $x\in U$.

Consider the case $X=\mathbb C^N$. A Zariski open set in $\mathbb C^N$ is the complement of a set $V = V(f_1,\dotsc,f_k)$ of common zeros of a collection of polynomials $f_1,\dotsc,f_k$ in $N$ variables.

Thus, to verify that some property $\mathcal P$ holds generically on $\mathbb C^N$, we first have to find such a set $V$ with the property that for all $x$, if $\mathcal P(x)$ does not hold then $x\in V$. This verifies that $\mathcal P(x)$ holds for all $x\in U$, where $U = \mathbb C^N \setminus V$. We also have to verify that $U$ is non-empty, which amounts to finding a specific element $x_0$ such that $x_0\not \in V$.

Note that $\dim V$ is at most $N-1$, which justifies the term ``generic''. In particular it is expected that a point $x\in \mathbb C^N$ taken at random\footnote{Say, according to the multivariate normal distribution} will lie in $U$.

Suppose that $\mathcal Q$ is another property of the points of $X$ and we want to show that both $\mathcal P(x)$ and $\mathcal Q(x)$ hold generically on $X$. Then it is enough to show separately that $\mathcal P(x)$ holds for generic $x$ and that $\mathcal Q(x)$ holds for generic $x$. This follows from the fact that the intersection of two non-empty Zariski open sets $U_1,U_2$ is always a nonempty Zariski open set. In practice, this means that after finding $U_1$ and $U_2$ it is enough to find two separate elements $x_1\in U_1$ and $x_2\in U_2$, which could be easier than finding an element $x_0\in U_1\cap U_2$.

In this article, the notion of a property holding generically is important for two reasons.
First, it is needed for the definition of the ML-degree. Indeed, the number (with multiplicities) of solutions $(\hat{x},\hat{y})$ to the score equations $\tilde\ell_i(x,y)$ given an empirical covariance matrix $S$ could vary with $S$. Nevertheless, it is constant for generic $S$, which justifies the use of a single number.
Second, we consider a family of models $\mathcal M_{A,B}$ parametrized by pairs of symmetric matrices $(A,B)$ and compute its ML-degree only for generic $A,B$. To perform the computation, we use several properties that hold for generic $A,B$ and $S$. We prove this separately for each one of them and use them together at the same time, as explained above.

\section{Geometry of the Score Equations}

In this section, we use B\'ezout's Theorem to derive a formula for computing $\deg (J)$. Lemma \ref{Lem:EmptyVarieties} will be used throughout the paper for all arguments involving generic $A, B$ and $S$. We will use Euler's homogeneous function theorem, which says that if $H(x,y)$ is a homogeneous function of degree $m$, then $mH = x H_x + y H_y$. We will also use the following fact about binary forms.

\begin{prop}\label{Prop:LinearForms}
Let $H(x,y) \in \C[x,y]$ be a homogeneous polynomial in two variables. Then $H(x,y)$ factors as a product of linear forms.
\end{prop}

\begin{proof}
Suppose that $H(x,y)$ is homogeneous of degree $d$. Let $h(x) := H(x,1)$ have degree $c$.
Since $\C$ is algebraically closed, by the Fundamental Theorem of Algebra, 
$h(x) = a \prod_{i=1}^c (x - r_i)$ for some $a, r_1, \dots, r_c \in \C$.
Then $H(x,y) = a y^{d-c} \prod_{i=1}^c (x - r_i y)$.
\end{proof}

We further note that a generic $h \in \C[x,y]$ factors as a product of \emph{distinct} linear forms. A binary form has a multiple root if and only if its discriminant vanishes, which is a closed condition on the space coefficients \cite[Sec.~0.12]{popovvinberg}.

\begin{lemma}\label{Lem:EmptyVarieties} For generic $A,B$ and $S$, the following projective varieties are empty:
\begin{enumerate}
\item $V(P,P_x),  V(P,T), V(P,P_y)$
\item $V(P_x,P_y)$
\item $V(P_x, T_x), V(P_y,T_y), V(T,T_x), V(T,T_y)$
\end{enumerate}
\end{lemma}
\begin{proof}
The emptiness of the varieties in the statement is an open condition in the space of parameters $(A,B,S)$.
For instance, the subset of the parameter space $\mathbb{A}_{(A,B,S)}$ where $V(P,P_x)$ is non-empty is the image of the variety defined by $P$ and $P_x$ in the space $\mathbb{A}_{(A,B,S)}\times \mathbb{P}^1_{[x:y]}$ under the first projection. This is a Zariski-closed subset of the parameter space by the projective elimination theorem \cite[Ch. 8.5]{cox1992}.

To show that the projective varieties in the statement are empty, we show that the polynomials defining them have no common factors. This makes use of Proposition \ref{Prop:LinearForms}, which states that every homogeneous form in two variables factors as a product of linear forms.

First, consider the case where $A$ is the $n \times n$ identity matrix, $B$ is the diagonal matrix with diagonal entries $1, \dots, n$, and $S = uu^T$ where $u$ is the vector of all ones. We have
\[
P = \prod_{k=1}^n (x + ky) \quad \text{and} \quad
 P_x = \sum_{k=1}^n \prod_{j\neq k}(x + jy).
\]
From this we deduce that if $p = x+ky$ is a linear form that divides $P$, then it does not divide $P_x$. This shows that $V(P,P_x)$ is empty.
The variety $V(P,T)$  is empty as well since $P_x = T$ in this case. Similarly, one shows that $V(P,P_y)$ is empty.

Euler's homogeneous function theorem applied to $P$ says that $nP = xP_x + yP_y$. Since $V(P,P_x)$ is generically empty, the same holds for $V(P_x, P_y)$.

To prove the rest of the statements, we switch to an element $(A,B,S)$ that makes the form of $T$ particularly simple.
Let $A$ and $B$ be as before and $u = (1,0,\dotsc,0)$.
This is allowed when combining generic properties as we explained at the end of Section~\ref{generic-properties}.
In this case we have
\[T = \prod_{k\neq 1}(x + ky) \quad \text{and}\quad
P_x = T + (x+y)T_x.
\]
Assume $p$ divides $P_x$ and $T_x$. Then $p$ divides $T$, hence we may assume $p=x + ky$ with $k\neq 1$. However, we have
$p\nmid P_x$ as before. This contradiction shows that $V(P_x,T_x)$ is empty. Similarly, $V(P_y, T_y)$ is empty. This example also has $T$ with no common roots, hence $V(T,T_x)$ and $V(T,T_y)$ are generically empty.
\end{proof}

Now we will show that the projective curves defined by $F$ and $G$ satisfy the hypothesis of B\'ezout's theorem; that is, that they do not share a common component. This justifies our application of B\'ezout's theorem and allows us to count the points in their variety. To prove this, we show that the polynomials $f$ and $g$ in \eqref{Eqn:fg} generically are irreducible and do not share a common factor.

\begin{lemma}\label{Lem:Irreducible}
The polynomials $f$ and $g$ in \eqref{Eqn:fg} are irreducible for generic $A, B$ and $S$. 
\end{lemma}
\begin{proof}
We prove the statement for $f$. The proof for $g$ is analogous. Write $f = F_{2n-1} + F_{2n-2}$, where
\[
F_{2n-1} = P P_x\quad\text{and}\quad F_{2n-2} = P T_x - T P_x
\]
and $\deg(F_i)=i$. If $f$ decomposes into a product of two polynomials, then at least one of them is homogeneous and we call it $h$.
Indeed,
otherwise the degrees of $F_{2n-1}$ and $F_{2n-2}$ would be at least two apart, when in fact they differ by one. Since $h$ is homogeneous and divides a nonzero sum of homogeneous polynomials, $h$ divides each of the summands $F_{2n-1}$ and $F_{2n-2}$. Using Proposition~\ref{Prop:LinearForms}, let $h_0$ be a linear factor of $h$. Since $h_0$ divides $F_{2n-1}$ and is irreducible, $h_0$ divides $P$ or $P_x$. In the first case, since $h_0$ divides $F_{2n-2}$, it would have to divide either $T$ or $P_x$. This would imply that one of the projective varieties $V(P,T)$ and $V(P,P_x)$ is nonempty. By Lemma~\ref{Lem:EmptyVarieties} this does not happen generically. In the second case, it would have to divide either $P$ or $T_x$, which for the same reason does not happen generically.
\end{proof}

\begin{lemma}\label{Lem:NotConstantMults}
For generic $A$, $B$ and $S$, the polynomials $f$ and $g$ in \eqref{Eqn:fg} are not constant multiples of one another.
\end{lemma}
\begin{proof}
If $f$ and $g$ are constant multiples of each other, then so are their highest degree terms $PP_x$ and $PP_y$. This does not happen generically since by Lemma~\ref{Lem:EmptyVarieties} the projective variety $V(P_x,P_y)$ is generically empty.
\end{proof}

\color{black}

Furthermore, we can describe exactly which points are removed from the affine variety $V(f,g)$ after we saturate with the determinant. For generic parameters, the only point that is removed after saturation is the origin.

\begin{prop}\label{Prop:SaturateOrigin}
For generic $A$, $B$ and $S$, we have
\[
	V(f,g)\setminus V(\det \Sigma) = V(f,g)\setminus \{(0,0)\}.
\]
\end{prop}

\begin{proof}
Let $q \in V(P, f, g)$. Then $f(q) = - T(q) P_x(q)$ and $g(q) = - T(q) P_y(q)$. In order to have $f(q) = g(q) = 0$, we must either have both $P_x(q) = P_y(q) = 0$ or $T(q) = 0$. By Lemma \ref{Lem:EmptyVarieties}, for generic $A, B$ and $S$, both of these imply $q= (0,0)$.
\end{proof}

\begin{prop}\label{Prop:RealRoadmapProp}
For generic $A$ and $B$, the \emph{ML}-degree of the model $\mathcal{M}_{A,B}$ is
\[
(2n-1)^2 - I_{[0:0:1]}(F,G) - \sum_{q\in V(F,G,z)}I_q(F,G).
\]
\end{prop}
\begin{proof}
The ML-degree of $\mathcal M_{A,B}$ is defined as the degree of the ideal $J = \langle f,g\rangle : (\det\Sigma)^\infty$. The affine variety $V(J)$ embeds in projective space as
\[
V(F,G)\setminus (V(F,G,z)\cup V(\det\Sigma)).
\]
By B\'ezout's Theorem (Theorem \ref{Thm:Bezout}), Lemmas \ref{Lem:Irreducible} and \ref{Lem:NotConstantMults} imply that the variety $V(F,G)$ is zero-dimensional.
Using Proposition \ref{Prop:SaturateOrigin}
we have
\begin{align*}
\deg (J)
%
%
&=
\sum_{\substack{q\in V(F,G)\\
q\notin \{[0:0:1]\} \cup V(F,G,z)}}I_q(F,G) \\
&= \sum_{q\in V(F,G)}I_q(F,G) - I_{[0:0:1]}(F,G) - \sum_{q\in V(F,G,z)}I_q(F,G).
\end{align*}
Both $F$ and $G$ have degree $2n-1$.
Applying Theorem~\ref{Thm:Bezout} to $F$ and $G$ gives the desired equality.
\end{proof}

\section{Multiplicity at the Origin}\label{Sec:MultiplicityOfTheOrigin}

In this section we compute the intersection multiplicity of the polynomials $f,g$ in \eqref{Eqn:fg} at the origin, denoted by $I_{[0:0:1]}(F,G)$  and also $I_{(0,0)}(f,g)$.

For a polynomial in two variables $h$ there is a notion of \emph{multiplicity} of $h$ at the origin, denoted $m_{(0,0)}(h)$. This is the degree of the lowest-degree summand in the decomposition of $h$ as a sum of homogeneous polynomials (for details, see \cite[Section 3.1]{fulton1989}). Since the polynomials $f,g$ can be written as the sum of a homogeneous degree $2n-2$ form with a homogeneous degree $2n-1$ form, we have $m_{(0,0)}(f) = m_{(0,0)}(g) = 2n-2$. We have the identity
\begin{equation}\label{Eq:FultonIdentity}
	I_{(0,0)}(f,g) = m_{(0,0)}(f) \cdot m_{(0,0)}(g)
\end{equation}
if the lowest-degree homogeneous forms of $f$ and $g$ share no common factors \cite[Section 3.3]{fulton1989}.
The degree $2n-2$ parts of $f$ and $g$ are $Q = PT_x - TP_x$ and $R = PT_y - TP_y$, respectively.

\begin{prop}\label{Prop:QR}
For generic $A$, $B$ and $S$, the polynomials $Q$ and $R$ share no common factor.
\end{prop}
\begin{proof}
By the definition of $Q$ and $R$ and two applications of Euler's homogeneous function theorem we have
\begin{align*}
xQ + yR &= (xT_x+yT_y)P - (xP_x+yP_y)T\\
&= (2n-2)TP - (2n-1)PT \\
&= -PT.
\end{align*}
Assume that $Q$ and $R$ share a common factor $p$, which we may assume is irreducible.
Then $p$ divides $PT$. So $p$ divides $P$ or $p$ divides $T$, but not both by Lemma~\ref{Lem:EmptyVarieties}. If $p$ divides $P$, then since $Q =  PT_x - TP_x$ and $p$ is a factor of $Q$, $p$ also divides $TP_x$. Similarly if $p$ divides $T$, then $p$ also divides $PT_x$. But then either $P$ and $TP_x$ share a common factor, or $T$ and $PT_x$ do. Each of the resulting four further cases does not occur generically by Lemma~\ref{Lem:EmptyVarieties}.
\end{proof}
\color {black}

\begin{cor}\label{Thm:OriginMult}
For generic $A, B$ and $S$, the intersection multiplicity of $f$ and $g$ at the origin is $(2n-2)^2$.
\end{cor}
\begin{proof}
By Proposition~\ref{Prop:QR}, this follows from~\eqref{Eq:FultonIdentity}.
\end{proof}

\section{Multiplicity at Infinity}\label{Sec:MultiplicityAtInfinity}

In this section we compute the intersection multiplicity at a point at infinity for the curves $V(f)$ and $V(g)$ defined by the polynomials in \eqref{Eqn:fg} for generic $A,B$ and $S$. To do this we use the connection between intersection multiplicity of curves and their series expansions about an intersection point.

Consider an irreducible polynomial $h$ in two variables such that $h(0,0)=0$ and $h_y(0,0)\neq 0$. By \cite[Sec. 7.11, Cor. 2]{fischer2001}, there exists an infinite series $\alpha=\sum_{\ind=1}^\infty a_\ind t^\ind$ and an open neighborhood $U\subset \C$ containing $t=0$ such that $h(t,\alpha(t))=0$ for all $t\in U$. The series $\alpha$ is called the \emph{series expansion} of $h$ at the origin.
The \textit{valuation} of a series is the number $M$ such that $a_M\neq 0$ and $a_m=0$ for all $m<M$.

\begin{prop} \label{Prop:ContactOrderMult}
Let $\fc$ and $\gc$ be irreducible polynomials in two variables such that $\fc$ and $\gc$ vanish at $(0,0)$ and $\fc_y$ and $\gc_y$ do not. Let $\alpha$ and $\beta$ be infinite series expansions of $\fc$ resp.\ $\gc$ at $(0,0)$. The intersection multiplicity $I_{(0,0)}(\fc,\gc)$ is the valuation of the series $\alpha-\beta$.
\end{prop}
\begin{proof}
By \cite[Sec. 8.7]{fischer2001}, the intersection multiplicity of $\fc$ and $\gc$ at $(0,0)$ is the valuation of the infinite series $\fc(t,\beta(t))$. We prove that this is the same as the valuation of $\alpha-\beta$. First, let
$s(t) = \sum_{\ind=1}^\infty s_\ind t^\ind$ be any infinite series and write $\fc = \sum_{i,j} c_{i,j} x^i y^j$, where the sum ranges over the pairs $(i,j)$ with $0<i+j\leq\deg(\fc)$. We have
\begin{align*}
	\fc(t,s(t)) = \sum_{i,j} c_{i,j} t^i \left(\sum_{\ind=1}^\infty s_\ind t^\ind\right)^j
	&= \sum_{i,j} c_{i,j} t^i \left(\sum_{\nu = 0}^\infty\left(\sum_{|a| = \nu} s_{a_1}\dotsm s_{a_j}\right) t^{\nu}\right)\\
	&= \sum_{i,j}\sum_{\nu = 0}^\infty\sum_{|a| = \nu} c_{i,j} s_{a_1}\dotsm s_{a_j} t^{\nu+i},
\end{align*}
The coefficient $r_\ind$ of $t^\ind$ in this infinite series is a finite sum of products of the form $c_{i,j} s_{a_1}\dotsm s_{a_j}$ with $a_j \leq \ind$ and $|a| + i = \ind$. The term $s_\ind$ only appears in $r_\ind$ when $j=1$ and $i=0$. Hence, we have $r_\ind = c_{0,1} s_\ind + p(s_1,\dotsc, s_{\ind-1})$ for some polynomial $p$, where $c_{0,1}\neq 0$ since $\fc_y(0,0)\neq 0$. For example, the coefficient $r_0$ is zero since $\fc,\gc$ vanishing at the origin implies that $c_{0,0}$ and $s_0$ are zero, and the coefficient of the first non-zero term is given by $r_1=c_{0,1}s_1+c_{1,0}$.

Write $\alpha(t)=\sum_{\ind=1}^\infty a_\ind t^\ind$ and $\beta(t) = \sum_{\ind=1}^\infty b_\ind t^\ind$. Suppose that the valuation of the series $\alpha - \beta$ is $\kval$. Then $a_\kval\neq b_\kval$ and $a_\ind=b_\ind$ for all $\ind <\kval$. We show that this is equivalent to $\fc(t,\beta (t)) = \sum_{\ind=1}^{\infty} r_{\ind} t^{\ind}$ having valuation $\kval$. Suppose that $\kval=1$; then $a_1\neq b_1$. Since $\fc(t,\alpha(t))$ is identically zero in a neighborhood of $t=0$, we have $r_\ind(a_1,\hdots,a_\ind)=0$ for all $\ind$. In particular
$r_1(a_1)=c_{0,1}a_1+c_{ 1, 0}=0$.
Since $a_1\neq b_1$ this implies that
$r_1(b_1)=c_{0,1}b_1+c_{ 1, 0}\neq 0$
and $\fc(t,\beta(t))$ has valuation one. Similarly if $\fc(t,\beta(t))$ has valuation one, then $r_1(a_1)\neq r_1(b_1)$ implying $a_1\neq b_1$. Thus $\alpha-\beta$ has valuation one if and only if $\fc(t,\beta(t))$ has valuation one.

Now suppose $\kval >1$. By the form of $r_\ind$ it now follows from an inductive argument on $\ind$ that $a_\ind$ and $b_\ind$ agree up to $\ind=\kval$ and differ at $\ind= \kval+1$ if and only if $r_\ind(a_1,\dotsc, a_\ind)$ and $r_{\ind}(b_1,\dotsc, b_\ind)$ agree up to $\ind=\kval$ and differ at $\ind = \kval+1$. Since $r_\ind(a_1,\dotsc,a_\ind) = 0 $ for all $\ind$, the latter is equivalent to $\fc(t,\beta (t))$ having valuation $\kval$.
\end{proof}
\begin{remark}\label{Rem:Shift}
In the context of Proposition~\ref{Prop:ContactOrderMult}, consider instead polynomials $\fc$ and $\gc$ defining the curves $\X$ resp.\ $\Y$ such that $\X$ and $\Y$ meet at a non-singular point $q$. Also, let $v$ be a vector such that the directional derivatives $\fc_v$ and $\gc_v$ do not vanish at $q$. Choose an affine-linear transformation $\varphi\colon \mathbb C^2\to \mathbb C^2$ sending $(0,0)$ to $q$ and $(0,1)$ to $v$. Then
$I_q(\fc,\gc) = I_{(0,0)}(\fc\circ \varphi, \gc\circ \varphi)$ and the polynomials $\fc\circ \varphi, \gc\circ \varphi$ satisfy the hypotheses of Proposition~\ref{Prop:ContactOrderMult}. Thus we can compute the intersection multiplicity at any non-singular intersection point of $\fc,\gc$ using Proposition~\ref{Prop:ContactOrderMult}.
\end{remark}

\begin{remark}
When the series $\alpha - \beta$ has valuation $\kval$, one says that $\fc$ and $\gc$ have \emph{contact order} or \emph{order of tangency} $\kval-1$ at $q$. Therefore the contact order of two curves at an intersection point is always one less than the intersection multiplicity. For more on contact order of algebraic curves see \cite[Chapter 5]{rutter2000}.
\end{remark}

\begin{remark}
The fact that the curves $\X$ and $\Y$ have intersection multiplicity one at $q$ if and only if the gradients of $\fc$ and $\gc$ at $q$ are linearly independent (see e.g. \cite[Sec 3.3]{fulton1989}) arises as a special case of Proposition~\ref{Prop:ContactOrderMult} once one computes the first terms of the series $\alpha$ and $\beta$.
\end{remark}

Returning to the expressions in \eqref{Eqn:fg}, recall that $F$ and $G$ denote the homogenizations of $f$ and $g$ with respect to the new variable $z$. The intersection points of $V(f)$ and $V(g)$ at infinity are exactly the variety $V(F,G,z)$. 

\begin{lemma}\label{Lem:ExactlyNPts}
For generic $A$, $B$ and $S$, the projective variety $V(F,G,z)$ consists of $n$ points of the form $[q_1:q_2:0]$ such that $P(q_1,q_2)=0$.
\end{lemma}

\begin{proof}
Let $q=[q_1:q_2:0]$ be a projective point of $V(F,G)$. We have
\begin{align*}
F&=PP_x+z(PT_x-TP_x)\\
G&=PP_y+z(PT_y-TP_y),
\end{align*}
and hence $V(F,G,z)$ consists of points $q$ where $[q_1{\,:\,}q_2]\in V(PP_x,PP_y)$. Clearly if $P(q_1,q_2)=0$, then $q\in V(F,G,z)$. These are the only such points since, by Lemma~\ref{Lem:EmptyVarieties}, for generic $A,B$ and $S$ the variety $V(P_x,P_y)$ is empty. By Proposition~\ref{Prop:LinearForms} $P(x,y)$ factors in $n$ linear forms. These forms are distinct, since a repeated factor would divide both $P$ and $P_x$, while $V(P,P_x)$ is empty by Lemma~\ref{Lem:EmptyVarieties}. Thus there are $n$ distinct points in $V(F,G,z)$.
\end{proof}

\begin{lemma}\label{Lem:EmptyVarietyProjPoints}
For generic $A,B$ and $S$, the projective variety $V(P, P_yT_x-P_xT_y)$ is empty.
\end{lemma}

\begin{proof}
Let $H = P_yT_x - P_xT_y.$ By applying Euler's homogeneous function theorem twice in the following chain of equalities, we have
\[
nT_x P - yH = T_x(nP-yP_y)+yP_xT_y = P_x (yT_y + xT_x) = (n-1)P_x T.
\]
If $P$ and $H$
have an irreducible common factor $p$, then $p \mid P_x T$.
This does not happen generically by Lemma~\ref{Lem:EmptyVarieties}.
\end{proof}

\begin{lemma}\label{Lem:MultTwo}
For generic $A,B$ and $S$, if $q\in V(F,G,z)$ then $I_q(F,G)=2$.
\end{lemma}

\begin{proof}
By Lemma \ref{Lem:ExactlyNPts}, such points are of the form $q=[q_1:q_2:0]$ where $P(q_1,q_2)=0$. Fix such a point $q$ and assume for simplicity that $q_1\neq 0$. This is not a restriction since the conditions $q_1=0$ and $P(q)=0$ imply $\det(B)=0$ which is a closed condition on the parameter space. Thus we can assume $q$ is of the form $[1:q_2:0]$.

Since intersection multiplicity at a point is a local quantity, we may dehomogenize with respect to $x$ and consider the intersection multiplicity of the affine curves $V(F(1,y,z))$ and $V(G(1,y,z))$ at $q$. We can compute the partial derivatives with respect to $y$ and $z$:
\begin{align}\label{Eq:FirstDerivative}
&F_y(x,y,z)=P_yP_x+PP_{xy}+z\left(\frac{d}{dy}(PT_x-TP_x)\right), &F_z(x,y,z)=PT_x-TP_x,\nonumber\\
&G_y(x,y,z)=P_y^2+PP_{yy}+z\left(\frac{d}{dy}(PT_y-TP_y)\right), &G_z(x,y,z)=PT_y-TP_y.
\end{align}
Consider the translated polynomials obtained by translating $q$ to $[1:0:0]$ given by $\tilde{F}=F(1,y+q_2,z)$ and $\tilde{G}=G(1,y+q_2,z)$. Then $\tilde{F}_z(1:0:0), \tilde{G}_z(1:0:0) \neq 0$ if and only if $F_z(q),G_z(q)\neq 0$, and from \eqref{Eq:FirstDerivative}, we have that 
$$
F_z(q)=(-TP_x)(1,q_2)\quad  \text{and}\quad G_z(q)=(-TP_y)(1,q_2).
$$
Since $P(1,q_2) = 0$, Lemma \ref{Lem:EmptyVarieties} implies that $F_z(q),G_z(q)\neq 0$. Thus there exist series expansions
$\alpha  = \sum_{\ind=1}^\infty a_\ind t^\ind$ and
$\beta= \sum_{\ind=1}^\infty b_\ind t^\ind$
such that, for all $t$ in a neighborhood of $t=0$,
$$
\tilde{F}\left(1,t,\sum_{\ind=1}^\infty a_\ind t^\ind\right)=0\quad \text{and} \quad \tilde{G}\left(1,t,\sum_{\ind=1}^\infty b_\ind t^\ind\right)=0,
$$
and hence
\begin{equation}\label{Eq:SeriesExp}
F\left(1,t+q_2,\sum_{\ind=1}^\infty a_\ind t^\ind \right)=0\quad \text{and} \quad G\left(1,t+q_2,\sum_{\ind =1}^\infty a_\ind t^\ind \right)=0,
\end{equation}
in this same neighborhood. Since $I_{[1:0:0]}(\tilde{F},\tilde{G})=I_q(F,G)$, by Proposition~\ref{Prop:ContactOrderMult} we can compute the valuation of the series $\alpha-\beta$ to determine $I_q(F,G)$. Differentiating the expressions in \eqref{Eq:SeriesExp} with respect to $t$, then substituting $t=0$ yields
$$
F_y(q)+F_z(q)a_1=0\quad \text{and}\quad G_y(q)+G_z(q)b_1=0.
$$
Thus $a_1=\frac{-F_y(q)}{F_z(q)}$ and $b_1=\frac{-G_y(q)}{G_z(q)}$, and $(F_yG_z-F_zG_y)(q)=0$ implies that $a_1-b_1=0$. By differentiating \eqref{Eq:SeriesExp} twice with respect to $t$ and substituting these values for $a_1$ and $b_1$, one can similarly show that
\begin{align}\label{Eq:SecDerivative}
a_2&=\left.\left(\frac{-F_{yy}F_z^2+2F_{yz}F_yF_z-F_{zz}F_y^2}{2F_z^3}\right)\right|_{q}\nonumber\\
b_2&=\left.\left(\frac{-G_{yy}G_z^2+2G_{yz}G_yG_z-G_{zz}G_y^2}{2G_z^3}\right)\right|_q.
\end{align}

Since we know that $a_1-b_1=0$, the valuation of $\alpha - \beta = \sum_{\ind =1}^\infty (a_\ind -b_\ind )t^\ind$ is \textit{at least} two; we now show that the valuation is  \emph{exactly} two for generic $A$, $B$ and $S$. We verified that $a_2-b_2\neq 0$ with the help of the computer algebra system \emph{Maple}~\cite{maple} by the following steps. First, we computed all second-order derivatives of $F$ and $G$ with respect to $y$ and $z$ in terms of partial derivatives of $P$ and $T$, by differentiating the expressions in \eqref{Eq:FirstDerivative}. Then, we substituted $P=0$ and $z=0$ in these expressions, which corresponds to evaluation at $q$. Thus from \eqref{Eq:SecDerivative} we can obtain expressions for $a_2$ and $b_2$ evaluated at $q$ in terms of partial derivatives of $P$ and $T$. Next we cleared denominators in the resulting expression for $a_2-b_2$, yielding
$$
(a_2-b_2)(q) = (T^4P_x^2P_y^4(P_yT_x-P_xT_y))(1,q_2).
$$
Since $P(1,q_2)=0,$ this expression does not generically evaluate to $0$ by Lemmas~\ref{Lem:EmptyVarieties} and~\ref{Lem:EmptyVarietyProjPoints}.
\end{proof}

\begin{cor}\label{Cor:IntersectionAtInfinity}
For generic  $A, B$ and $S$, we have $\sum_{q\in V(F,G,z)}I_q(F,G)=2n.$ 
\end{cor}

\begin{proof}
This follows from Lemmas \ref{Lem:ExactlyNPts} and \ref{Lem:MultTwo}.
\end{proof}

Now we can prove our main result that $\deg(J) = 2n-3$:

\begin{proof}[Proof of Theorem~\ref{Thm:Main}]
Combining Proposition~\ref{Prop:RealRoadmapProp} with Corollaries~\ref{Thm:OriginMult} and~\ref{Cor:IntersectionAtInfinity} shows that the ML-degree of $\mathcal M_{A,B}$ for generic $A$ and $B$ is
\[
	(2n-1)^2 - (2n-2)^2 - 2n = 2n-3. \qedhere
\]
\end{proof}

\section{Discussion}

In \cite{sturmfels2019}, Sturmfels, Timme and Zwiernik use numerical algebraic geometry methods implemented in the Julia package \texttt{LinearGaussianCovariance.jl} to compute the ML-degrees of linear Gaussian covariance models for several values of $n$ and $m$, where $n$ is the size of the covariance matrix and $m$ is the dimension of model. We have proven that for $m=2$ and arbitrary $n$, the ML-degree is $2n-3$, which agrees with the computations in Table 1 of \cite{sturmfels2019}.

For higher dimensional linear spaces, where $m>2$, the score equations consist of the partial derivatives of $\tilde\ell$ with respect to the $m$ parameters of the linear space. Again, in this case, these are rational functions of the data and the parameters. For instance when $m=3$, we can consider the linear span of three $n\times n$ matrices $A,B$, and $C$. Then if $\Sigma=xA+yB+zC$, $P=\det\Sigma$ and $T=\tr(S\adj\Sigma)$, the score equations are
\begin{align*}
\tilde\ell_x(x,y,z) &= \frac{P_x}{P} + \frac{P T_x - T P_x}{P^2} \\
\tilde\ell_y(x,y,z) &= \frac{P_y}{P} + \frac{P T_y - T P_y}{P^2}\\
\tilde\ell_z(x,y,z) &= \frac{P_z}{P} + \frac{P T_z - T P_z}{P^2},
\end{align*}
and we can similarly define polynomials
\begin{align*}
f(x,y,z) &:= P P_x + P T_x - T P_x \\
g(x,y,z) &:= P P_y + P T_y - T P_y\\
h(x,y,z) &:= P P_z + P T_z - T P_z,
\end{align*}
such that the ML-degree of the model is the degree of $J=\mathcal{I}(f,g,h) :\langle \det\Sigma\rangle^\infty$. The authors of \cite{sturmfels2019} conjecture that the ML-degree in this case is $3n^2-9n+7$. To prove this conjecture as we did for $m=2$, one might turn to a higher dimensional generalization of B\'ezout's Theorem, which says that the number of solutions to $V(f,g,h)$ counted with multiplicity is the product $\deg(f)\deg(g)\deg(h)$ \emph{provided that $V(f,g,h)$ is zero-dimensional} (see for example \cite[Sec. 3, Ch. 3]{cox1998} or \cite[Sec. 2.1, Ch. 3]{shafarevich2013}). This zero-dimensionality restriction is necessary for equality, otherwise the product of the degrees in this case simply gives an upper bound for the number of zero-dimensional solutions counted with multiplicity \cite[Thm. 12.3]{fulton1998}.

Indeed the variety $V(f,g,h)$ contains the one-dimensional affine variety $V(P,T)$ as well as a ``curve at infinity'' corresponding to the vanishing of $P$. When $m=2$, the variety $V(P,T)\subset \C^2$ consisted of only the origin and the elements at infinity were points whose multiplicity we were able to ascertain using properties of curves. This illustrates the added difficulties in counting solutions when moving from planar intersection theory to higher dimensional intersections. 

The authors of \cite{sturmfels2019} also consider the \emph{generic diagonal model}, in which the linear space that comprises the model consists of diagonal matrices. Their computations show that for $m=2$, the ML-degree of the generic diagonal model for the first several values of $n$ is also $2n-3$, see \cite[Table 2]{sturmfels2019}. It follows from the proof of our result that this ML-degree is indeed $2n-3$ for all $n$, as the witnesses for the non-emptiness of the open dense sets that we produced in the proof of Lemma~\ref{Lem:EmptyVarieties} were all diagonal matrices.  For $m > 2$ and $n > 3$, the ML-degree of the generic diagonal model is conjectured in  \cite{sturmfels2019} to be strictly less than the corresponding generic linear Gaussian covariance model. This suggests that the study of linear Gaussian covariance models of arbitrary dimension will require us to look beyond diagonal matrices as witnesses to the non-triviality of some open conditions. 

Indeed, many of the projective varieties in Lemma \ref{Lem:EmptyVarieties} are nonempty for diagonal matrices when $m>2$. 
For example, when $m \geq 3$, the determinant of $P$ for a diagonal $\Sigma$ has a nonempty singular locus. 
Let $m = 3$ and let
\[
\Sigma = xA + yB + zC
\]
where $A$, $B$, and $C$ are the diagonal matrices with diagonal entries $(a_1, \dots, a_n)$, $(b_1,\dots, b_n)$ and $(c_1,\dots,c_n)$, respectively.
Then we have
\[
P = \prod_{i=1}^n (a_ix + b_iy +c_iz).
\]
The derivatives of $P$ are of the form
\[
P_x = \sum_{i=1}^n a_i \prod_{\substack{j=1 \\ j \neq i}}^n (a_j x + b_j y + c_jz),
\]
and similarly for $P_y$ and $P_z$.
So any projective point in the intersection of linear spaces of the form
\[
V(a_ix + b_iy +c_iz) \cap V(a_j x + b_j y + c_jz)
\]
for $i \neq j$ is a singular point of $P$.
When $m > 2$, these intersections are nonempty, so such singular points exist.

Thus, when $\Sigma$ is not defined by diagonal matrices, the problem of finding witnesses to the emptiness of the varieties in Lemma \ref{Lem:EmptyVarieties} for arbitrary $n$ is non-trivial, which adds another layer of difficulty for establishing the ML-degree when $m > 3$. Nevertheless we believe that examining the structure of the score equations for $m=2$ provides a possible blueprint for approaching the problem for $m > 2$, although it will require different tools from intersection theory.

For the purposes of statistical inference, one is most interested in \emph{real} solutions to the score equations,
as these are the ones that may have statistical meaning.
Furthermore, it would be nice to understand whether there are truly $2n-3$ distinct (complex) solutions to the score equations,
as opposed to some having higher multiplicity. Based on the examples we have computed, we conjecture an affirmative answer. We thus still have the following open questions regarding the $m=2$ case.

\begin{prob}
How many real solutions can the score equations of a generic two-dimensional linear Gaussian covariance model have?
\end{prob}

\begin{conj}
For generic values of $A,B$ and $S$, the score equations of $\mathcal{M}_{A,B}$ with sample covariance matrix $S$ have
$2n-3$ \emph{distinct} solutions.
\end{conj}

\section*{Acknowledgements}
The authors would like to thank Carlos Amendola, Irina Kogan, Emre Sert\"oz, Bernd Sturmfels, Seth Sullivant, Sascha Timme, Caroline Uhler, Cynthia Vinzant and Piotr Zwiernik for many helpful conversations. We would also like to thank the anonymous reviewers for their detailed comments on the manuscript. All three authors were supported by the Max-Planck-Institute for Mathematics in the Sciences. Jane Coons was supported by the US National Science Foundation (DMS 1615660).

\bibliographystyle{acm}
\bibliography{small-covariance}

\vfill
\footnotesize \textbf{Authors' affilitiations:}\\
Jane Ivy Coons, North Carolina State University
\hfill {\tt jicoons@ncsu.edu}\\
Orlando Marigliano, MPI MiS Leipzig
\hfill {\tt orlando.marigliano@mis.mpg.de}\\
Michael Ruddy, MPI MiS Leipzig
\hfill  {\tt michael.ruddy@mis.mpg.de}

\end{document}